\documentclass[11pt]{article}
\usepackage{amsmath,amsthm,amssymb, amsfonts,accents,nccmath}
\usepackage{enumitem}
\usepackage{array}
\usepackage[english]{babel}
\usepackage{dsfont}
\usepackage{booktabs}
\usepackage[linesnumbered,commentsnumbered,ruled,vlined]{algorithm2e}
\usepackage{tikz}
\usepackage{tikz-network}

\makeatletter
\def\BState{\State\hskip-\ALG@thistlm}
\makeatother
\numberwithin{equation}{section}

\newtheorem{theorem}{Theorem}[section]
\newtheorem{cor}[theorem]{Corollary}
\newtheorem{lem}[theorem]{Lemma}
\newtheorem{prop}[theorem]{Proposition}
\newtheorem{defn}[theorem]{Definition}
\newtheorem{rem}[theorem]{Remark}
\newtheorem{ex}[theorem]{Example}

\newenvironment{itemizeReduced}{
\begin{list}{\labelitemi}{\leftmargin=1em}
\setlength{\itemsep}{1pt}
\setlength{\parskip}{0pt}
\setlength{\parsep}{0pt}}{\end{list}
}

\newcommand{\bigzero}{\mbox{\normalfont\Large\bfseries 0}}

\title{Distance Matrix of a  Class of Completely Positive Graphs: Determinant and Inverse}
\author{Joyentanuj Das$^*$,  \quad Sachindranath Jayaraman$^*$ \quad and \quad Sumit Mohanty\footnote{School of Mathematics, IISER Thiruvananthapuram, Maruthamala P.O., Vithura, 
Thiruvananthapuram,\newline \indent   Kerala- 695 551, India.
 \newline \indent Emails:
joyentanuj16@iisertvm.ac.in, \; sachindranathj@gmail.com, sachindranathj@iisertvm.ac.in \; sumit@iisertvm.ac.in,  sumitmath@gmail.com }}
\date{}

\begin{document}

\maketitle

\begin{abstract}
A real symmetric matrix $A$ is said to be completely positive if it can be written as $BB^t$ for some 
(not necessarily square) nonnegative matrix $B$. A simple graph $G$ is called a completely positive 
graph if every doubly nonnegative matrix realization of $G$ is a completely positive matrix. Our aim 
in this manuscript is to compute the determinant and inverse (when it exists) of the distance matrix of 
a class of completely positive graphs. Similar to trees, we obtain a  relation for the inverse 
of the distance matrix of a class of completely positive graphs involving the 
Laplacian matrix, a rank one matrix and a matrix $\mathcal{R}$. We also determine the eigenvalues of 
some matrices related to the matrix $\mathcal{R}$.
\end{abstract}


\section{Introduction and Preliminaries } \label{sec:intro and prelim}

Let $G=(V,E)$  be a  finite, connected, simple and undirected graph with $V$ as the  set of vertices and $E\subset V \times V$ as the set of edges. We write $i\sim j$ to indicate that the vertices $i,j$ are adjacent in $G$. The degree of the vertex $i$ is denoted by $\delta_i$. A graph with $n$ vertices is called complete, if each vertex of the graph is adjacent to every other vertex and is denoted by $K_n$. A graph $G=(V,E)$ said to be  bipartite if $V$ can be partitioned into two subsets $V_1$ and $V_2$ such that $E\subset V_1 \times V_2$. A bipartite graph $G=(V,E)$ with  the  partition $V_1$ and $V_2$  is said to be a complete bipartite graph, if every vertex in $V_1$ is adjacent to every vertex of $V_2$. The complete bipartite graph with $|V_1|=m$ and $|V_2|=n$ is denoted by $K_{m,n}$.

The distance $d(i,j)$ from $i$ and $j$ in $G$ is the length of the shortest path from $i$ and $j$. The distance matrix of graph  $G$  is an  $n \times n$ matrix,  denoted by  $D(G) = [d_{ij}]$, where
$d_{ij}=d(i,j)   \text{, if} \ i \neq j,$ and $0  \text{ if} \ i = j.$ The Laplacian matrix of $G$  is an $n \times n$ matrix, denoted as $L(G)=[l_{ij}]$, where
$
l_{ij}=
\delta_i  \text{ if} \ i = j, 
-1   \text{ if} \ i \neq j, i \sim j $ and $0 \text{ otherwise}$.

Let $T$ be a tree with $n$ vertices. In~\cite{Gr1}, the authors proved that the determinant of the distance matrix $D(T)$ of $T$ is given by $\det D(T)=(-1)^{n-1}(n-1)2^{n-2}.$  Note that, the determinant does  not depend on the structure of the tree but the number of vertices. In \cite{Gr2}, it was shown that the inverse of the distance matrix of a tree is given by $D(T)^{-1} = -\dfrac{1}{2}L(T) + \dfrac{1}{2(n-1)}\tau \tau^t,$ where, $\tau = (2-\delta_1,2-\delta_2,...,2-\delta_n)^t.$  The above expression gives a formula for inverse of distance matrix of a tree in terms of the
Laplacian  matrix.

The determinant and the inverse of the distance matrix were also studied for bi-directed trees and weighted trees 
(for details, see~\cite{Bp2, Zhou}). In~\cite{Bp1}, similar results were studied for q-analogue of the distance matrix, which is a generalization of the distance matrix for a tree. The inverse of the distance matrix has been explored for graphs such as block graphs, bi-block graphs and cactoid digraph (for details, see~\cite{Bp3, Hou1, Hou3}). In this article, we study the determinant and inverse of the distance matrix for certain classes of completely positive graphs.

\subsection{Completely Positive Graphs}

A matrix is said to be doubly nonnegative if it is both entrywise nonnegative and positive semidefinite. A real symmetric matrix $A$ is called \emph{completely positive (cp-matrix)} if it can be decomposed as $A = BB^t$, where $B$ is a (not necessarily square) nonnegative matrix. Given an  $n \times n$ symmetric matrix $A$, we can associate a graph $G(A)$ with the matrix. The vertices of $G(A)$ are $1,2,\ldots,n$ and $i\sim j$  in $G(A)$  if and only if $i\neq j$ and $a_{ij}\neq 0$.
\begin{defn}
A graph $G$ is said to be a \emph{completely positive graph (cp-graph)},  if every doubly nonnegative, symmetric matrix $A$ whose graph is $G$ is completely positive.
\end{defn}

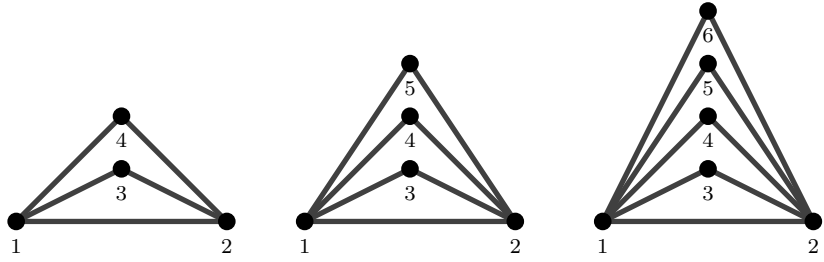
\begin{figure}
\centering
\begin{tikzpicture}[scale=.7]
\Vertex[label=1,position=below,size=.1,color=black]{A} 
\Vertex[x=4,label=2,position=below,size=.1,color=black]{B} 
\Vertex[x=2,y=1,label=3,position=below,size=.1,color=black]{C} 
\Vertex[x=2,y=2,label=4,position=below,size=.11,color=black]{D}
\Edge[lw=2pt](A)(B)
\Edge[lw=2pt](A)(C)
\Edge[lw=2pt](A)(D)
\Edge[lw=2pt](B)(C)
\Edge[lw=2pt](B)(D)
\end{tikzpicture}\hspace{5mm} 
\begin{tikzpicture} [scale=.7] 
\Vertex[label=1,position=below,size=.1,color=black]{A} 
\Vertex[x=4,label=2,position=below,size=.1,color=black]{B} 
\Vertex[x=2,y=1,label=3,position=below,size=.1,color=black]{C} 
\Vertex[x=2,y=2,label=4,position=below,size=.11,color=black]{D}
\Vertex[x=2,y=3,label=5,position=below,size=.1,color=black]{E} 
\Edge[lw=2pt](A)(B)
\Edge[lw=2pt](A)(C)
\Edge[lw=2pt](A)(D)
\Edge[lw=2pt](A)(E)
\Edge[lw=2pt](B)(C)
\Edge[lw=2pt](B)(D)
\Edge[lw=2pt](B)(E)
\end{tikzpicture}  \hspace{5mm}  \begin{tikzpicture}[scale=.7]
\Vertex[label=1,position=below,size=.1,color=black]{A} 
\Vertex[x=4,label=2,position=below,size=.1,color=black]{B} 
\Vertex[x=2,y=1,label=3,position=below,size=.1,color=black]{C} 
\Vertex[x=2,y=2,label=4,position=below,size=.11,color=black]{D}
\Vertex[x=2,y=3,label=5,position=below,size=.1,color=black]{E} 
\Vertex[x=2,y=4,label=6,position=below,size=.1,color=black]{F}
\Edge[lw=2pt](A)(B)
\Edge[lw=2pt](A)(C)
\Edge[lw=2pt](A)(D)
\Edge[lw=2pt](A)(E)
\Edge[lw=2pt](A)(F)
\Edge[lw=2pt](B)(C)
\Edge[lw=2pt](B)(D)
\Edge[lw=2pt](B)(E)
\Edge[lw=2pt](B)(F)
\end{tikzpicture}
\caption{$T_4,T_5,T_6$}\label{fig:M1}
\end{figure}

A vertex $v$ of a graph $G$ is a \emph{cut vertex} of $G$ if $G - v$ is disconnected. A block of the graph $G$ is a maximal connected subgraph of $G$ that has no cut-vertex. There are many equivalent conditions to prove that a graph $G$ is completely positive. We state a few of these.  For $n \geq 3$, let $T_n$ be a graph  consisting of $(n-2)$ triangles and a common base (see Figure~\ref{fig:M1}). Some of the equivalent conditions are stated below.

\begin{theorem}\label{thm:cpg1}\cite[Corollary 2.6]{Br}
The following properties of a graph $G$ are equivalent:
\begin{itemizeReduced}
    \item [(i)] G is a cp-graph.
    \item [(ii)]Each block of $G$ is a cp-graph.
    \item [(iii)]Each block of $G$ is either bipartite, or a $K_4$, or a $T_n$.
\end{itemizeReduced}
\end{theorem}

A graph said to be a bi-block graph if each of its blocks is a complete bipartite graph. Note that every tree is a bipartite graph.  The determinant and inverse of the distance matrix have been studied for these graphs, and interestingly, the formula for the inverse of the distance matrix comes in terms of the Laplacian matrix (or the Laplacian like matrix; for details, 
see~\cite{Gr1,Gr2,Hou3}).  In view of Theorem~\ref{thm:cpg1}, both bi-block graphs and trees are cp-graphs. In this 
article, our primary interest is to compute the determinant and inverse of the distance matrix of a class of cp-graphs such that each of its blocks is $T_n$,  for a fixed $n$ and with a central cut vertex, which is not a base vertex. We denote such graph as $T_n^{(b)}$, where $b \ (\geq 2)$ represent number of blocks (see Figure~\ref{fig:T-n-b}). 

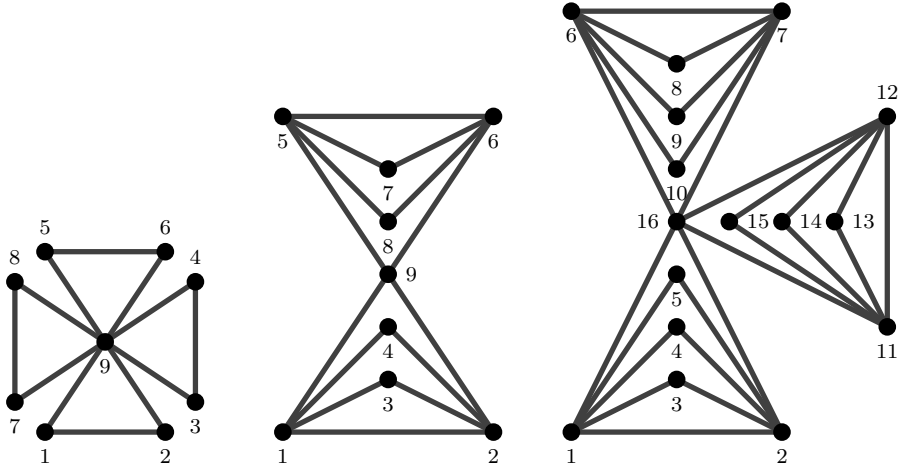
\begin{figure}
\centering
\begin{tikzpicture}[scale=0.4]
\Vertex[label=1,position=below,size=.1,color=black]{A} 
\Vertex[x=4,label=2,position=below,size=.1,color=black]{B} 
\Vertex[x=2,y=3,label=9,position=below,size=.1,color=black]{C}

\Vertex[x=0,y=6,label=5,position=above,size=.1,color=black]{A1} 
\Vertex[x=4,y=6,label=6,position=above,size=.1,color=black]{B1}

\Vertex[x=5,y=1,label=3,position=below,size=.1,color=black]{A2} 
\Vertex[x=5,y=5,label=4,position=above,size=.1,color=black]{B2}

\Vertex[x=-1,y=1,label=7,position=below,size=.1,color=black]{A3} 
\Vertex[x=-1,y=5,label=8,position=above,size=.1,color=black]{B3}

\Edge[lw=2pt](A)(B)
\Edge[lw=2pt](A)(C)
\Edge[lw=2pt](B)(C)

\Edge[lw=2pt](A1)(B1)
\Edge[lw=2pt](A1)(C)
\Edge[lw=2pt](B1)(C)

\Edge[lw=2pt](A2)(B2)
\Edge[lw=2pt](A2)(C)
\Edge[lw=2pt](B2)(C)

\Edge[lw=2pt](A3)(B3)
\Edge[lw=2pt](A3)(C)
\Edge[lw=2pt](B3)(C)
\end{tikzpicture} \hspace{5mm} \begin{tikzpicture}[scale=0.7]
\Vertex[label=1,position=below,size=.1,color=black]{A} 
\Vertex[x=4,label=2,position=below,size=.1,color=black]{B} 
\Vertex[x=2,y=1,label=3,position=below,size=.1,color=black]{C} 
\Vertex[x=2,y=2,label=4,position=below,size=.11,color=black]{D}
\Vertex[x=2,y=3,label=9,position=right,size=.1,color=black]{E}

\Vertex[x=0,y=6,label=5,position=below,size=.1,color=black]{A1} 
\Vertex[x=4,y=6,label=6,position=below,size=.1,color=black]{B1}
\Vertex[x=2,y=5,label=7,position=below,size=.1,color=black]{C1} 
\Vertex[x=2,y=4,label=8,position=below,size=.1,color=black]{D1}

\Edge[lw=2pt](A)(B)
\Edge[lw=2pt](A)(C)
\Edge[lw=2pt](A)(D)
\Edge[lw=2pt](A)(E)
\Edge[lw=2pt](B)(C)
\Edge[lw=2pt](B)(D)
\Edge[lw=2pt](B)(E)

\Edge[lw=2pt](A1)(B1)
\Edge[lw=2pt](A1)(C1)
\Edge[lw=2pt](A1)(D1)
\Edge[lw=2pt](A1)(E)
\Edge[lw=2pt](B1)(C1)
\Edge[lw=2pt](B1)(D1)
\Edge[lw=2pt](B1)(E)

\end{tikzpicture}\hspace{5mm} 
\begin{tikzpicture}[scale=.7]
\Vertex[label=1,position=below,size=.1,color=black]{A} 
\Vertex[x=4,label=2,position=below,size=.1,color=black]{B} 
\Vertex[x=2,y=1,label=3,position=below,size=.1,color=black]{C} 
\Vertex[x=2,y=2,label=4,position=below,size=.11,color=black]{D}
\Vertex[x=2,y=3,label=5,position=below,size=.1,color=black]{E} 
\Vertex[x=2,y=4,label=16,position=left,size=.1,color=black]{F}

\Vertex[x=0,y=8,label=6,position=below,size=.1,color=black]{A1} 
\Vertex[x=4,y=8,label=7,position=below,size=.1,color=black]{B1}
\Vertex[x=2,y=7,label=8,position=below,size=.1,color=black]{D1} 
\Vertex[x=2,y=6,label=9,position=below,size=.1,color=black]{E1} 
\Vertex[x=2,y=5,label=10,position=below,size=.1,color=black]{F1}

\Vertex[x=3,y=4,label=15,position=right,size=.1,color=black]{A2} 
\Vertex[x=4,y=4,label=14,position=right,size=.1,color=black]{A3} 
\Vertex[x=5,y=4,label=13,position=right,size=.1,color=black]{A4}
\Vertex[x=6,y=2,label=11,position=below,size=.1,color=black]{B2} 
\Vertex[x=6,y=6,label=12,position=above,size=.1,color=black]{B3}

\Edge[lw=2pt](A)(B)
\Edge[lw=2pt](A)(C)
\Edge[lw=2pt](A)(D)
\Edge[lw=2pt](A)(E)
\Edge[lw=2pt](A)(F)
\Edge[lw=2pt](B)(C)
\Edge[lw=2pt](B)(D)
\Edge[lw=2pt](B)(E)
\Edge[lw=2pt](B)(F)
\Edge[lw=2pt](A1)(B1)
\Edge[lw=2pt](A1)(D1)
\Edge[lw=2pt](A1)(E1)
\Edge[lw=2pt](A1)(F1)
\Edge[lw=2pt](A1)(F)
\Edge[lw=2pt](B1)(D1)
\Edge[lw=2pt](B1)(E1)
\Edge[lw=2pt](B1)(F1)
\Edge[lw=2pt](B1)(F)

\Edge[lw=2pt](B2)(B3)
\Edge[lw=2pt](B2)(A2)
\Edge[lw=2pt](B2)(A3)
\Edge[lw=2pt](B2)(A4)
\Edge[lw=2pt](B2)(F)

\Edge[lw=2pt](B3)(A2)
\Edge[lw=2pt](B3)(A3)
\Edge[lw=2pt](B3)(A4)
\Edge[lw=2pt](B3)(F)
\end{tikzpicture}
\caption{$T_3^{(4)},T_5^{(2)},T_6^{(3)}$}\label{fig:T-n-b} 
\end{figure}

Let $D(T_{n}^{(b)})$ denote the distance matrix of $T_{n}^{(b)}$. It turns out that, similar to the case of bi-block graphs and trees,  $D(T_{n}^{(b)})^{-1}$  can be expressed in terms of the Laplacian matrix $L(T_{n}^{(b)})$ and a new matrix $\mathcal{R}(T_n^{(b)})$. In particular, for $n\neq 6$ and $b\geq 2$, the expression for $D(T_n^{(b)})^{-1}$ is given by
\begin{equation}\label{eqn:Dinv}
D(T_n^{(b)})^{-1} = -\frac{1}{2} L(T_n^{(b)}) + \frac{1}{2b}J + \frac{1}{2(n-6)b}\mathcal{R}(T_n^{(b)}),
\end{equation}
where, $J$ is a matrix of all one's. Unlike the case of trees and bi-block graphs, we only give an expression for $D(T_n^{(b)})^{-1}$ similar to the trees. At this point, we are not able to retrieve information about the cp-graph from the matrix $\mathcal{R}$.

This article is organized as follows. In Section~\ref{sec:pre-matrix-notation}, we recall necessary results from matrix theory and also fix a few notations which will appear throughout this article. In Section~\ref{sec:Single block}, we compute the determinant and inverse of the distance matrix of single blocks of $T_{n}$, $K_{m,n}  (m=n\neq 2)$ and $K_4$. In Section~\ref{sec:Multiple block}, we compute the determinant of the distance matrix of $T_{n}^{(b)}$ and find its inverse, whenever it exists. The manuscript ends with a section on some properties of the matrix related to $\mathcal{R}(T_n^{(b)})$, appearing in Eqn~(\ref{eqn:Dinv}).

\subsection{Notations and Some Preliminary Results}\label{sec:pre-matrix-notation}

We begin this section by introducing a few notations which will be used throughout this article. Let $I_n$ and $ \mathds{1}_n, $ denote the identity matrix and the column vector of all ones of order $n$ respectively.  Further, $J_{m \times n}$  denotes the $m\times n$ matrix of all ones and if $m=n$, we use the notation $J_m$.  We write $\mathbf{0}_{m \times n}$ to represent zero matrix of order $m \times n$ and simply write $\mathbf{0}$. Unless there is a scope of confusion, we omit the order of the matrices.

Let $A$ be an $m \times n$ matrix. If $S_1 \subset \{1,\cdots , m\}, \ S_2\subset \{1,\cdots,n\}$, then $A[S_1 | S_2]$ 
denotes the submatrix of A determined by the rows corresponding to $S_1$ and the columns corresponding to $S_2$. If $S_1 = S_2 = S$, then we write $A[S]$ to denote the principal submatrix of $A$ determined by the set $S$. 
Given a matrix $A$, we use $A^t$ to denote the transpose of the matrix.

Let $B$ be an $n\times n$ matrix partitioned as
\begin{equation}\label{eqn:B}
B= \left[
\begin{array}{c|c}
B_{11}& B_{12} \\
\midrule
B_{21} &B_{22}
\end{array}
\right],
\end{equation}
where $ B_{11}$ and $B_{22}$ are square matrices. If $B_{11}$ is nonsingular, then the Schur complement of $B_{11}$ in $B$ is defined to be the matrix $B/B_{11}=B_{22} -B_{21}B_{11}^{-1}B_{12}$. The following 
result, which gives the inverse (whenever the matrix is invertible) of a partitioned matrix using Schur complements will be 
used in our calculations.


\begin{prop}\label{prop:schur}(\cite{Bapat}, page 4)
Let $B$ be a nonsingular matrix partitioned as in Eqn~(\ref{eqn:B}). Suppose  $B_{11}$ is square and invertible. Then
\[
B^{-1}=\left[
\begin{array}{c|c}
B_{11}^{-1}+ B_{11}^{-1} B_{12} (B/B_{11})^{-1} B_{21}B_{11}^{-1}& -B_{11}^{-1}B_{12}(B/B_{11})^{-1}  \\
\midrule
-(B/B_{11})^{-1}B_{21}B_{11}^{-1} & (B/B_{11})^{-1}
\end{array}
\right]
\]
\end{prop}

Given two matrices $A$ and $B$ of the same size such that $A$ and $B$ are invertible, it is not necessary that $A+B$ is invertible. However, when both $A$ and $A+B$ are invertible, and $B$ has rank $1$, Miller obtained a formula for the inverse of $A+B$. This result is stated below.


\begin{theorem}\label{thm:A+B-inverse}(\cite{Miller})
Let $A$ and $A+B$ be invertible, and $B$ have rank $1$. If $g=\text{trace}(BA^{-1})$, then $g \neq -1$ and 
$$(A+B)^{-1}=A^{-1} - \frac{1}{1+g} A^{-1}BA^{-1}.$$
\end{theorem}

We end this section with a standard result on computing the determinant of block triangular matrices. The proof 
of this result can be found in any standard text on matrix theory.

\begin{prop}\label{prop:blockdet}
 Let $A_{11}
   \mbox{ and } A_{22}$ be square matrices.  Then $$\det \left[
\begin{array}{c|c}
A_{11} & \bigzero  \\
\midrule
A_{21} & A_{22}
\end{array}
\right] = \det A_{11}\times \det A_{22}.$$
\end{prop}

%

\section{ Determinant and Inverse of Single Blocks}\label{sec:Single block}

We begin this section with the study of the distance matrix of a single block of $T_n$. We find an expression for the determinant of the distance matrix of $T_n$, $D(T_n)$, establishing that $\det D(T_n)\neq 0$ and provide the formula for the inverse of the same.

We discuss the choice of vertex indexing for the graph $T_n$ as it is essential to our proofs. We recall that for $n \geq 3$, $T_n$ is the graph with $n$ vertices consisting of $n-2$ triangles with a common base. Let $\{1,2,\ldots,n\}$ be the vertex set of $T_n$ and we denote $1,2$ as  the base vertices and $3,4,\ldots,n$ as non-base vertices (for example see Figure~\ref{fig:M1}). Throughout this article, unless stated otherwise, we will follow the above vertex indexing for  $T_n$. The following 
lemma and the remark that follows will be used in computing the inverse of $D(T_n)$. The proofs are skipped.

\begin{lem}\label{lem:aI+bJ}
Let  $n \geq 2$, let $J_n$, $I_n$ be matrices as defined before. For $a \neq 0$, the eigenvalues of $aI_n+bJ_n$ are $a$ 
and $a+nb$ with multiplicities $n-1$ and $1$, respectively, and the determinant is given by $a^{n-1}(a+nb)$. 
Moreover, the matrix is invertible if and only if $a+nb \neq 0$ and the inverse is given by 
$$(aI_n + bJ_n)^{-1} = \frac{1}{a} \left(I_n - \frac{b}{a+nb} J_n\right).$$
\end{lem}

\begin{rem}\label{rem-new}
Let  $\mathcal{A}_2 = \left[\begin{array}{cc}
                            0 & 1\\
                            1 & 0
                           \end{array}
                      \right]$ and let $J_{r \times s}$ be the $r\times s$ matrix of all ones. Then, $\mathcal{A}_2^2=I_2$, 
$\mathcal{A}_2J_{2\times s }=J_{2\times s }$, $J_{r\times 2 }\mathcal{A}_2= J_{r\times 2 }$, $\mathcal{A}_2J_{2 }\mathcal{A}_2=J_{2 }$ and $J_{r\times t }J_{t\times s }= t J_{r\times s }.$
\end{rem}

The distance matrix of $T_n$ can be written in the following block matrix form,
\begin{equation}\label{eqn: D(T_n)}
D(T_n)=
\left[
\begin{array}{c|c}
\mathcal{A}_2 & J_{2\times (n-2)} \\
\midrule
J_{(n-2)\times 2} &2(J_{n-2}-I_{n-2})
\end{array}
\right].
\end{equation}

The following theorem gives the determinant and inverse of $D(T_n)$, when $n \geq 3$.


\begin{theorem}\label{thm:det_Tn}
Let $D(T_n)$ be the distance matrix of the graph $T_n$ for $n \geq 3$. Then the determinant and inverse of $D(T_n)$ 
equal $(-1)^{n-1}  2^{n-2}$ and \[
D(T_n)^{-1}=\left[
\begin{array}{c|c}
\mathcal{A}_2-\frac{n-2}{2}J_2 & \frac{1}{2}J_{2\times (n-2)} \\
\midrule
\frac{1}{2}J_{(n-2)\times 2 } & -\frac{1}{2}I_{n-2}
\end{array}
\right],
\] respectively.
\end{theorem}

\begin{proof}
We first add the first two columns of the distance matrix $D(T_n)$ and  subtract this from  third column onwards to get 
a matrix of the form
$$
\left[
\begin{array}{c|c}
\mathcal{A}_2 & \mathbf{0}_{2\times (n-2)} \\
\midrule
J_{(n-2)\times 2} &-2I_{n-2}
\end{array}
\right].
$$ 
Proposition~\ref{prop:blockdet} yields the result for the determinant of $D(T_n)$. Note that, for $n \geq 3$, 
$\det D(T_n)$ is non-zero, and consequently the inverse of $D(T_n)$ exists. To calculate the inverse 
of $D(T_n)$, we use Schur complements. The Schur complement of $\mathcal{A}_2$ in $D(T_n)$ is given by  
$D(T_n)/\mathcal{A} = 2(J_{n-2}-I_{n-2}) - J_{(n-2)\times 2}\mathcal{A}_2^{-1}J_{(n-2)\times 2 }$. 
Using Lemma~\ref{lem:aI+bJ}, Remark~\ref{rem-new} and Proposition~\ref{prop:schur}, we have
\begin{align*}
D(T_n)^{-1}=&  \left[
\begin{array}{c|c}
\mathcal{A}_2^{-1} + \mathcal{A}_2^{-1}J_{2\times (n-2)}(D(T_n)/\mathcal{A}_2)^{-1})J_{(n-2)\times 2 }\mathcal{A}_2^{-1} & \mathcal{A}_2^{-1}J_{2\times (n-2)}(D(T_n)/\mathcal{A}_2)^{-1} \\
\midrule
- (D(T_n)/\mathcal{A}_2)^{-1} J_{(n-2)\times 2 } \mathcal{A}_2^{-1} & (D(T_n)/\mathcal{A}_2)^{-1}
\end{array}
\right]  \\
   =& \left[
\begin{array}{c|c}
\mathcal{A}_2-\frac{n-2}{2}J_2 & \frac{1}{2}J_{2\times (n-2)} \\
\midrule
\frac{1}{2}J_{(n-2)\times 2 } & -\frac{1}{2}I_{n-2}
\end{array}
\right].
\end{align*}
\end{proof}

Similar to the result for trees, the next corollary  gives  another representation for the $D(T_n)^{-1}$ in terms of the Laplacian matrix $L(T_n)$ of $T_n$, which is given by

$$L(T_n)=\left[
\begin{array}{c|c}
(n-1)I_2- \mathcal{A}_2 & -J_{2\times (n-2)} \\
\midrule
- J_{(n-2)\times 2 } & 2I_{n-2}
\end{array}
\right]$$

\begin{cor}
Let $D(T_n)$ and $L(T_n)$ be the distance and Laplacian matrices, respectively, of the graph $T_n$. Then,
$$
D(T_n)^{-1} = -\frac{1}{2} L(T_n) + \frac{1}{2} J_n + \frac{1}{2} \mathcal{R}(T_n),
$$
where
\begin{equation}\label{eqn:R_nT_n}
\mathcal{R}(T_n)=\left[
\begin{array}{c|c}
-(n-2) \mathcal{A}_2 & -J_{2\times (n-2)} \\
\midrule
- J_{(n-2)\times 2 } & I_{n-2} - J_{n-2}
\end{array}
\right].
\end{equation}

\end{cor}

The following remark is interesting in its own right. The general case will be discussed in Section \ref{sec:Eigen}.

\begin{rem}
Consider the matrix $\mathcal{R}(T_n)$. Let $B = \{1,2\}$ be the set of base vertices and $N = \{3,4,\cdots,n\}$ be 
the set of non-base vertices of $T_n$. It is easy to compute the eigenvalues of the principal submatrices 
$\mathcal{R}(T_n)[B]$ and $\mathcal{R}(T_n)[N]$ of $\mathcal{R}(T_n)$, corresponding to the vertex sets $B$ and $N$. 
Since the eigenvalues of $\mathcal{A}_2$ are $\pm 1$, it follows that the eigenvalues of $\mathcal{R}(T_n)[B]$ are 
$\mp (n-2)$, each with multiplicity $1$. Further, by Eqn.~(\ref{eqn:R_nT_n}),  $\mathcal{R}(T_n)[N] = I_{n-2} - J_{n-2}$. 
It follows from Lemma \ref{lem:aI+bJ} that the eigenvalues of $\mathcal{R}(T_n)[N]$ are $1-n$ and $1$ with 
multiplicities $1$ and $n-3$, respectively.
\end{rem}

%

We now find the determinant and inverse of the distance matrix of a complete bipartite graph. These results have appeared in the work of Hou et al~\cite{Hou3}, but the proofs and techniques given in this article are completely different than theirs. 
Let $K_{m,n}$ be the complete bipartite graph with $m + n$ vertices. Let $\{1,2,\ldots,m\}$ and $\{m+1,m+2,\ldots,m+n\}$ be the vertex set partition of $K_{m,n}$. Then the distance matrix of $K_{m,n}$ is in the following block form:
\begin{small}
\begin{equation}\label{D_Kmn}
D(K_{m,n})=\left[
\begin{array}{c|c}
2(J_m-I_m) & J_{m\times n} \\
\midrule
J_{n \times m} & 2(J_n-I_n)
\end{array}
\right].
\end{equation}
\end{small}

\begin{theorem}\label{Thm:KmnDet}
Let $D(K_{m,n})$ be the distance matrix of $K_{m,n}$. Then the determinant of $D(K_{m,n})$ is given by
$$\det D(K_{m,n}) = (-2)^{m+n-2} \times \left(4(m-1)(n-1) - mn \right).$$
Moreover, $\det D(K_{m,n})= 0$ if and only if $m=n=2$.
\end{theorem}

\begin{proof}
Starting with the distance matrix of $K_{m,n}$ as in Eqn~\eqref{D_Kmn} we  use elementary row or column operations to obtain the determinant. For each partition we subtract the first column from the remaining columns and then
add all the rows to the first row. After applying the above steps on $D(K_{m,n})$, the resulting  matrix $\widetilde{D}(K_{m,n})$ has the following block matrix form:
\begin{small}
$$\widetilde{D}(K_{m,n})= \left[
\begin{array}{c|c|c|c}
2(m-1) & \mathbf{0}_{1 \times (m-1)} & m &  \mathbf{0}_{1 \times (n-1)} \\
  \hline
2 \mathds{1}_{(m-1) \times 1} &  -2 I_{m-1} & \mathds{1}_{(m-1) \times 1} &  \mathbf{0}_{(m-1) \times (n-1)}\\
  \midrule
n & \mathbf{0}_{1 \times (m-1)}& 2(n-1)& \mathbf{0}_{1 \times (n-1)} \\
  \hline
\mathds{1}_{(n-1) \times 1} & \mathbf{0}_{(n-1) \times (m-1)}& 2 \mathds{1}_{(n-1) \times 1}& -2 I_{n-1}
\end{array}
\right].$$
\end{small}
We repartition the matrix $\widetilde{D}(K_{m,n})$ so that
$\widetilde{D}(K_{m,n})= \left[
\begin{array}{c|c}
\widetilde{A}_{11} & \bigzero\\
  \midrule
* &  \widetilde{A}_{22}
\end{array}
\right],$ where
\begin{small}
$$\widetilde{A}_{11} =  \left[
\begin{array}{c|c|c}
2(m-1)                        & \mathbf{0}_{1 \times (m-1)} & m\\
  \midrule
2 \mathds{1}_{(m-1) \times 1} &  -2 I_{m-1}        & \mathds{1}_{(m-1) \times 1}\\
 \midrule
n                             & \mathbf{0}_{1 \times (m-1)} &  2(n-1)
\end{array}
\right]$$
\end{small}
and $\widetilde{A}_{22} = -2 I_{n-1}$. The $\det \widetilde{A}_{11} = (-2)^{m-1} \times \left(4(m-1)(n-1) - mn \right),$ by expanding along the first row. Thus, using Proposition~\ref{prop:blockdet}, we get
$$\det \widetilde{D}(K_{m,n}) = (\det \widetilde{A}_{11})( \det \widetilde{A}_{22}) = (-2)^{m-1} \times \left(4(m-1)(n-1) - mn \right) \times (-2)^{n-1}.$$
It is easy to see that the necessary and sufficient  condition  for  $\det D(K_{m,n})=0$ is $m=n=2.$
\end{proof}

\begin{lem}\label{lemma:star}
Let $n \neq 1$ and $D(K_{n,1})$ be the distance matrix of the star graph $K_{n,1}$. Then,
\begin{small}
 \[
D(K_{n,1})^{-1}=\left[
\begin{array}{c|c}
 \dfrac{1}{2n}J_n-\dfrac{1}{2}I_n & \dfrac{1}{n} \mathds{1}_{n \times 1} \\
  \midrule
\dfrac{1}{n} \mathds{1}_{1 \times n} &  - \dfrac{2(n-1)}{n}
\end{array}
\right].
\]
\end{small}
\end{lem}

\begin{proof}
By Lemma \ref{lem:aI+bJ}, the matrix $2(J_n-I_n)$ is invertible, and therefore the Schur complement of $2(J_n-I_n)$ in 
$D(K_{n,1})$ is given by $D(K_{n,1})/2(J_n-I_n) = - \dfrac{n}{2(n-1)}$. Then by using Proposition \ref{prop:schur}, 
the required result follows.
\end{proof}

\begin{lem}\label{lemma:Kmn}
 Let $m,n \neq 1$ and $m=n\neq2$. Let $D(K_{m,n})$ be the distance matrix of the complete bipartite graph $K_{m,n}$. Then,
 \begin{small}
  \[
D(K_{m,n})^{-1}=\left[
\begin{array}{c|c}
 \dfrac{3n-4}{2(3mn-4(m+n-1))}J_m-\dfrac{1}{2}I_m & -\dfrac{1}{3mn-4(m+n-1)}J_{m\times n} \\
  \midrule
-\dfrac{1}{3mn-4(m+n-1)}J_{n\times m} &  \dfrac{3m-4}{2(3mn-4(m+n-1))}J_n-\dfrac{1}{2}I_n
\end{array}
\right].
\] 
 \end{small}
\end{lem}

\begin{proof}
By Lemma \ref{lem:aI+bJ}, the matrix $2(J_n-I_n)$ is invertible and so the Schur complement of $2(J_n-I_n)$ in 
$D(K_{m,n})$ is given by
\begin{align}\label{eqn:schur-bipartite}
D(K_{m,n})/2(J_n-I_n)=&  2(J_m-I_m)-J_{m\times n}[2(J_n-I_n)]^{-1}J_{n\times m}\\ \nonumber
                     =& 2(J_m-I_m)- J_{m\times n}\left[\frac{1}{2}\left(\frac{1}{n-1}J_n-I_n)\right)\right]
                        J_{n\times m}\\ \nonumber
                     =& 2(J_m-I_m)-\frac{n}{2(n-1)}J_m = A+B,
\end{align}
where $A=2(J_m-I_m)$ and $B=-\dfrac{n}{2(n-1)}J_m$.
Note that, $A+B=-2 I_m+ \dfrac{3n-4}{2(n-1)}J_m$ is of the form $aI_m + bJ_m$, where $a+mb = 0$ if and only if 
$m=n=2$. By Lemma~\ref{lem:aI+bJ}, $A+B$ is invertible if and only if  $m=n\neq 2$.

Note that $BA^{-1}= -\dfrac{n}{4(n-1)(m-1)}J_m$ and $ g=trace (BA^{-1})=-\dfrac{mn}{4(n-1)(m-1)}$. Therefore, 
$1+g= \dfrac{3mn-4(m+n-1)}{4(n-1)(m-1)}$ and $A^{-1}BA^{-1} = - \dfrac{n}{8(n-1)(m-1)^2}J_m$. Since, 
$1+g\neq 0$ for $m=n\neq 2$, we have from Eqn~(\ref{eqn:schur-bipartite}) and Theorem~\ref{thm:A+B-inverse}, 
that,
 $$(D(K_{m,n})/2(J_n-I_n))^{-1} = (A+B)^{-1} = \dfrac{3n-4}{2(3mn-4(m+n-1))}J_m-\frac{1}{2}I_m.$$
Finally, Proposition \ref{prop:schur}  and  Lemma~\ref{lem:aI+bJ} yields the required form of $D(K_{m,n})^{-1}$.
\end{proof}

\begin{theorem}
Let   $m=n\neq2$ and $D(K_{m,n})$ be the distance matrix of the complete bipartite graph $K_{m,n}$. Then,
\begin{small}
  \[
D(K_{m,n})^{-1}=\left[
\begin{array}{c|c}
 \dfrac{3n-4}{2(3mn-4(m+n-1))}J_m-\frac{1}{2}I_m & -\dfrac{1}{3mn-4(m+n-1)}J_{m\times n} \\
  \midrule
-\dfrac{1}{3mn-4(m+n-1)}J_{n\times m} &  \dfrac{3m-4}{2(3mn-4(m+n-1))}J_n-\frac{1}{2}I_n
\end{array}
\right].
\]
\end{small}
\end{theorem}

\begin{proof}
When $m=n = 1$, $D(K_{m,n}) = \mathcal{A}_2$. When $m,n > 1$ the result follows from Lemmas \ref{lemma:star} and \ref{lemma:Kmn}.
\end{proof}

\section{ Determinant and Inverse of $D(T_n^{(b)})$}\label{sec:Multiple block}

In this section, we compute the determinant of the distance matrix of $T_n^{(b)}$ and find a formula for the inverse 
whenever it exists. Similar to the single block case, the indexing of the vertices plays an important role in our 
results. We discuss the indexing of vertices of $T_n^{(b)}$ below. Recall that $T_n^{(b)}$ is a cp-graph consisting of 
$b$ blocks, where each of the block is $T_n$, with a central cut vertex (see Figure~\ref{fig:T-n-b}). Therefore, 
$|V(T_n^{(b)})| = b(n-1)+1$ and $|E(T_n^{(b)})| = b(2n-3)$.

Let $V(T_n^{(b)})=  \{1,2, \ldots,  b(n-1)+1\}$ be the vertex set of $T_n^{(b)}$, where $b(n-1)+1$ represents the 
cut vertex. Let us write $V(T_n^{(b)})$ as a disjoint union as follows.

\begin{eqnarray*}
V(T_n^{(b)})&= & \{(k-1)(n-1)+i : k=1,2,\ldots, b  \mbox{ and } i=1,2,\ldots, n-1 \} \cup \{ b(n-1)+1\} \nonumber\\
           & = &  \displaystyle \sqcup_{k=1}^{b} \{(k-1)(n-1)+i : i=1,2,\ldots, n-1\} \sqcup \{ b(n-1)+1\},
\end{eqnarray*} 
where $\sqcup$ denotes disjoint union. 
If $G_1,G_2,\ldots, G_b$ are the $b$ blocks of $T_n^{(b)}$, then for each block $G_k, \; 1\leq k\leq b$,
$$V(G_k)= \{(k-1)(n-1)+i : i=1,2,\ldots, n-1\} \cup \{ b(n-1)+1\}.$$
For each block, the first two vertices represent base vertices and the remaining represent the non-base vertices. Let $D(T_n^{(b)})$ be the distance matrix of $T_n^{(b)}$. Then,
\begin{small}
\begin{equation}\label{eqn:Det}
D(T_n^{(b)})=
\left[
\begin{array}{c|c|c|c|c|c|c}
D_1 & D_2 & D_2 &\cdots &D_2 &D_2 & \mathbf{d}_3 \\
\hline
D_2 & D_1 & D_2 &\cdots &D_2 &D_2 & \mathbf{d}_3 \\
\hline
D_2 & D_2 & D_1  &\cdots &D_2 &D_2 & \mathbf{d}_3 \\
\hline
\vdots & \vdots & \vdots & \ddots & \vdots & \vdots & \vdots\\
\hline
D_2 & D_2 & D_2 &\cdots &D_1 &D_2 & \mathbf{d}_3 \\
\hline
D_2 & D_2 & D_2 &\cdots &D_2 &D_1  & \mathbf{d}_3 \\
\hline
\mathbf{d}_3^t & \mathbf{d}_3^t & \mathbf{d}_3^t &\cdots &\mathbf{d}_3^t &\mathbf{d}_3^t & 0 \\
\end{array}
\right]  \ \mbox{ where}
\end{equation}
\end{small}


\noindent $\underline{\textbf{Case 1:}} \, $($n = 3$)
$$
D_1 =  \mathcal{A}_2, \ D_2 = 2J_2 \ \text{and} \
\mathbf{d}_3= \mathds{1}_2.
$$
\noindent $\underline{\textbf{Case 2:}} \, $($n \geq 4$)
$$
D_1 = \left[
\begin{array}{c|c}
\mathcal{A}_2& J_{2 \times (n-3)}\\
\hline
J_{(n-3) \times 2} & 2(J_{n-3}-I_{n-3})
\end{array}
\right],
D_2 = \left[
\begin{array}{c|c}
2J_2 & 3J_{2 \times (n-3)}\\
\hline
3J_{(n-3) \times 2} & 4J_{n-3}\\
\end{array}
\right] \ \text{and} \
\mathbf{d}_3= \left[
\begin{array}{c}
\mathds{1}_2 \\
\hline
2\mathds{1}_{n-3}
\end{array}
\right].$$

\begin{theorem}\label{thm:det_Tn_b}
Let $D(T_n^{(b)})$ be the distance matrix of $T_n^{(b)}$. Then the determinant of $D(T_n^{(b)})$ is given by 
$$\det D(T_n^{(b)}) = (-1)^{b(n-4)+1} \times 2^{b(n-3)+1} \times b \times (n-6)^{b-1}.$$
\end{theorem}

\begin{proof}
Our aim is to use elementary row or column operations to obtain the $\det D(T_n^{(b)})$. Let $\sigma = (2,(n-1)b+1)$ be the transposition and $P_{\sigma} = (p_{ij})$ be the corresponding permutation matrix. Let us label the block-rows and block-columns in \eqref{eqn:Det} by, $R_1,\cdots, R_b, R_{b+1}$ and $C_1,\cdots, C_b, C_{b+1}$, respectively. The algorithm is summarized in the steps below.

\begin{itemizeReduced}
\item[1.] First do $R_1 \to \sum_{i=1}^b R_i$, followed by $C_j \to C_j - C_1$ for $j = 2, \cdots b$.
\item[2.] For each $R_i$, $1\leq i \leq b$ add the second row to the first, and if $n > 3$ add rows $4, \cdots, n-1$ 
to the third row.
\item[3.] For each $C_j$, $1\leq j \leq b$ subtract the first column from the second, and if $n > 3$ subtract the third column from all the columns $4, \cdots, n-1$.
\item[4.] If $n>3$, for each $C_j$, $1\leq j \leq b$ subtract the third column from the first. Then add $4$ times the first and $2$ times the second column to the third column.
\item[5.] Rearrange the rows and columns by pre-multiplying and post-multiplying by the permutation matrix $P_{\sigma}$.
\end{itemizeReduced}

Note that in the above steps, the permutation matrix $P_{\sigma}$  is multiplied twice and will therefore not change the sign of the determinant. The resulting matrix obtained is given by

\begin{small}
\[
\widetilde{D}(T_n^{(b)})=
\left[
\begin{array}{c|c|c|c|c}
\widetilde{D}_1(T_n^{(b)}) & \mathbf{0} &\cdots &\mathbf{0} &\mathbf{0} \\
\hline
* & \widetilde{D}_2 (T_n^{(b)}) &\cdots &\mathbf{0} & \mathbf{0} \\
\hline
\vdots & \vdots & \ddots & \vdots & \vdots \\
\hline
* & * &\cdots &\widetilde{D}_2 (T_n^{(b)}) & \mathbf{0}\\
\hline
* & * &\cdots &* & -1 \\
\end{array}
\right]
\]
\end{small}
where, the matrices $\widetilde{D}_1(T_n^{(b)})$ and $\widetilde{D}_2 (T_n^{(b)})$ are of the following forms:

\begin{small}
\noindent $\underline{\textbf{Case 1:}} \, $($n = 3$)
\begin{equation*}
\widetilde{D}_1(T_3^{(b)})  = \left[
\begin{array}{cc}
4(b-1)+1 & 2b\\
1 & 0 \\
\end{array}
\right]  \mbox{ and }
\widetilde{D}_2(T_3^{(b)})  = \left[
\begin{array}{cc}
-3 & 0  \\
-1 & -1 \\
\end{array} \;
\right].
\end{equation*}

\noindent $\underline{\textbf{Case 2:}}\, $($n = 4$)

\begin{equation*}
\widetilde{D}_1(T_4^{(b)})  = \left[
\begin{array}{ccc}
4(b-1)+1 & 2b &2(3(b-1)+1)\\
1 & 0 & 2\\
(3b-2) & 2b & 4(b-1)\\
\end{array}
\right]  \mbox{ and }
\widetilde{D}_2(T_4^{(b)})  = \left[
\begin{array}{ccc}
1 & 0 &0 \\
1 & -1 & 0\\
2 & 0 & 4\\
\end{array} \;
\right].
\end{equation*}

\noindent $\underline{\textbf{Case 3:}}\, $($n \ge 5$)

\begin{eqnarray*}
  \widetilde{D}_1(T_n^{(b)}) &=&\left[
\begin{array}{ccc|c}
4(b-1)+1 & 2b &2(3(b-1)+1) &\\
1 & 0 & 2&\bigzero\\
(n-3)(3b-2) & 2b(n-3) & 4(n-3)(b-1)+2(n-4)&\\
\midrule
&*&&-2I_{n-4}
\end{array}
\right]  \mbox{ and }\\
  \widetilde{D}_2(T_n^{(b)}) &=& \left[
\begin{array}{ccc|c}
1 & 0 &0& \\
* & -1 & 0&\bigzero\\
* & * & -2(n-6)&\\
\midrule
\ & * &  &-2I_{n-4}
\end{array} \;
\right].
\end{eqnarray*}
\end{small}
The matrix $\widetilde{D}(T_n^{(b)})$ can be rewritten in the following  form
$ \left[
\begin{array}{c|c}
\widetilde{D}_1(T_n^{(b)}) & \bigzero \\
\hline
* & Y
\end{array}
\right], $
where $Y$ is a lower triangular matrix of the form
\begin{small}
\[
Y=
\left[
\begin{array}{c|c|c|c}
\widetilde{D}_2 (T_n^{(b)}) &\cdots &\mathbf{0} & \mathbf{0} \\
\hline
\vdots & \ddots & \vdots & \vdots \\
\hline
* &\cdots &\widetilde{D}_2 (T_n^{(b)}) & \mathbf{0}\\
\hline
* &\cdots &* & -1 \\
\end{array}
\right].
\]
\end{small}
Since each $\widetilde{D}_2 (T_n^{(b)})$ is a lower triangular matrix, using Proposition~\ref{prop:blockdet}, we have $\det D(T_n^{(b)})=\det \widetilde{D}_1(T_n^{(b)}) \times \det Y$ and hence the desired result follows.
\end{proof}

The above result shows that the $\det D(T_n^{(b)})$ depends only on the number of vertices in each block and the number of blocks. It can be seen by substituting $b=1$, the formula for $\det D(T_n)$ coincides with the one obtained in Theorem~\ref{thm:det_Tn}. Theorem~\ref{thm:det_Tn_b} implies that $\det D(T_n^{(b)})\neq 0$ when $n\neq 6$ and $b\geq 2$.  
To determine $D(T_n^{(b)})^{-1}$ when $n\neq 6$ and $b\geq 2$,  we first express the Laplacian matrix $L(T_n^{(b)})$ of $T_n^{(b)}$, which in block form is given by
\begin{equation}\label{eqn:LapT_nb}
L(T_n^{(b)})=
\left[
\begin{array}{c|c|c|c|c}
L_1        & \mathbf{0} &\cdots &\mathbf{0} &\textbf{l}_2 \\
\hline
\mathbf{0} & L_1        &\cdots &\mathbf{0} &\textbf{l}_2\\
\hline
\vdots     & \vdots     &\ddots & \vdots    & \vdots \\
\hline
\mathbf{0} & \mathbf{0} &\cdots &L_1        & \textbf{l}_2\\
\hline
\textbf{l}_2^t &\textbf{l}_2^t  &\cdots &\textbf{l}_2^t & 2b \\
\end{array}
\right],  \ \mbox{ where}
\end{equation}

\begin{small}
\noindent $\underline{\textbf{Case 1:}} \, $($n = 3$)
 $$L_1 =
\begin{array}{c}
2I_2- \mathcal{A}_2
\end{array},
\textbf{l}_2 =
\begin{array}{c}
-\mathds{1}_2\\
\end{array}$$

\noindent $\underline{\textbf{Case 2:}}\, $($n \ge 4$)

$$
L_1 = \left[
\begin{array}{c|c}
(n-1)I_2- \mathcal{A}_2             &-J_{2 \times (n-3)}\\
\hline
-J_{(n-3) \times 2}                 &2I_{n-3}\\
\end{array}
\right]  \mbox{ and }
\textbf{l}_2 = \left[
\begin{array}{c}
-\mathds{1}_2\\
\hline
\mathbf{0}_{(n-3) \times 1}\\
\end{array}
\right].
$$
\end{small}

Similar to the case of single blocks of $T_n$, we express the inverse of the distance matrix of $T_n^{(b)}$ in terms of the Laplacian matrix. Consider the matrix $\mathcal{R}(T_n^{(b)})$ in block matrix form:

\begin{small}
\begin{equation}\label{eqn:R_Tn_b}
\mathcal{R}(T_n^{(b)}) =
\left[
\begin{array}{c|c|c|c|c|c|c}
\mathcal{R}_1 & \mathcal{R}_2 & \mathcal{R}_2 &\cdots &\mathcal{R}_2 &\mathcal{R}_2 & \mathbf{r}_3 \\
\hline
\mathcal{R}_2 & \mathcal{R}_1 & \mathcal{R}_2 &\cdots &\mathcal{R}_2 &\mathcal{R}_2 & \mathbf{r}_3 \\
\hline
\mathcal{R}_2 & \mathcal{R}_2 & \mathcal{R}_1  &\cdots &\mathcal{R}_2 &\mathcal{R}_2 & \mathbf{r}_3 \\
\hline
\vdots & \vdots & \vdots & \ddots & \vdots & \vdots & \vdots\\
\hline
\mathcal{R}_2 & \mathcal{R}_2 & \mathcal{R}_2 &\cdots &\mathcal{R}_1 &\mathcal{R}_2 & \mathbf{r}_3 \\
\hline
\mathcal{R}_2 & \mathcal{R}_2 & \mathcal{R}_2 &\cdots  &\mathcal{R}_2 &\mathcal{R}_1 & \mathbf{r}_3 \\
\hline
\mathbf{r}_3^t & \mathbf{r}_3^t & \mathbf{r}_3^t &\cdots & \mathbf{r}_3^t & \mathbf{r}_3^t & r \\
\end{array}
\right], \ \mbox{ where}
\end{equation}

\noindent $\underline{\textbf{Case 1:}} \, $($n = 3$)
$$
\mathcal{R}_1 = -2(b-1)I_2+(b+2)\mathcal{A}_2,\
\mathcal{R}_2 = 2J_2,\
\mathbf{r}_3 = 3b \mathds{1}_2 \
\mbox{ and } \
r = -6(b-1)^2.
$$

\noindent $\underline{\textbf{Case 1:}} \, $($n \geq 4$)
\begin{eqnarray*}\label{eqn:R1,R2}
\mathcal{R}_1 &=& \left[
\begin{array}{c|c}
(n-5)(n-2)(b-1)I_2+(n-2)(b-(n-5))\mathcal{A}_2 & -((n-4)b-2)J_{2 \times (n-3)}\\
\hline
-((n-4)b-2)J_{(n-3) \times 2} & b(n-6)I_{n-3}+(b-(n-5))J_{n-3}
\end{array}
\right],\nonumber\\
\mathcal{R}_2 &=&
\left[
\begin{array}{c|c}
-(n-5)(n-2)J_2 & 2J_{2 \times (n-3)} \\
\hline
2J_{(n-3) \times 2} & -(n-5)J_{n-3}
\end{array}
\right], \;
\mathbf{r}_3 =  \left[
\begin{array}{c}
(n-6)\times [(n-4)b-(n-3)]\mathds{1}_2 \\
\hline
-b (n-6)\mathds{1}_{n-3}
\end{array}
\right] \mbox{and} \\
  r &=& -(n-5)(n-6)(b-1)^2.\nonumber
\end{eqnarray*}
\end{small}

\begin{theorem}\label{thm:invthm}
Let $n\neq 6$, $b\geq 2$ and  $D(T_n^{(b)})$ be the distance matrix of the graph $T_n^{(b)}$. Then,  the inverse of $D(T_n^{(b)})$ is given by

$$D(T_n^{(b)})^{-1} = -\frac{1}{2} L(T_n^{(b)}) + \frac{1}{2b}J + \frac{1}{2(n-6)b}\mathcal{R}(T_n^{(b)}),$$
where $L(T_n^{(b)})$ is the Laplacian matrix of $T_n^{(b)}$, $J$ is the matrix is of all ones with the conformal order and 
$\mathcal{R}(T_n^{(b)})$ is the matrix as defined in Eqn~(\ref{eqn:R_Tn_b}).
\end{theorem}

\begin{proof}
Let $$X = -\dfrac{1}{2} L(T_n^{(b)}) + \dfrac{1}{2b}J + \dfrac{1}{2(n-6)b} \ \mathcal{R}(T_n^{(b)}).$$ The block matrix 
form of $X$ is given by
\begin{small}
\[
X =
\left[
\begin{array}{c|c|c|c|c|c|c}
X_1 & X_2 & X_2 &\cdots &X_2 & X_2 & \mathbf{x}_3 \\
\hline
X_2 & X_1 & X_2 &\cdots &X_2 & X_2 & \mathbf{x}_3 \\
\hline
X_2 & X_2  & X_1 &\cdots &X_2 & X_2 & \mathbf{x}_3 \\
\hline
\vdots & \vdots & \vdots & \ddots & \vdots & \vdots & \vdots \\
\hline
X_2 & X_2 & X_2 &\cdots &X_1 & X_2  & \mathbf{x}_3 \\
\hline
X_2 & X_2 & X_2 &\cdots & X_2 &X_1 & \mathbf{x}_3 \\
\hline
\mathbf{x}_3^t & \mathbf{x}_3^t & \mathbf{x}_3^t &\cdots & \mathbf{x}_3^t & \mathbf{x}_3^t & x \\
\end{array}
\right],
\]
where,\\
\noindent $\underline{\textbf{Case 1:}} \, $($n = 3$)
$$
X_1 = -\dfrac{1}{6b}[(4b-1)J_2-6b\mathcal{A}_2],\
X_2 = \dfrac{1}{6b}J_2,\
\mathbf{x}_3 = \dfrac{1}{2b}\mathds{1}_2 \
\text{and}\
x = \dfrac{3-4b}{2b}.
$$
\noindent $\underline{\textbf{Case 2:}} \, $($n \geq 4$)
\begin{eqnarray*}
 X_1 &=& \dfrac{1}{2b(n-6)}\left[
\begin{array}{c|c}
(4b-(n-4)^2)J_2+2b(n-6)\mathcal{A}_2 & (n-4-2b)J_{2 \times (n-3)}\\
\hline
(n-4-2b) J_{(n-3) \times 2} & -(n-6)bI_{n-3}+(b-1)J_{n-3}
\end{array}
\right],\\
X_2
&=&  \dfrac{1}{2b(n-6)}
\left[
\begin{array}{c|c}
-(n-4)^2J_2 & (n-4)J_{2 \times (n-3)} \\
\hline
(n-4)J_{(n-3) \times 2} & -J_{n-3}
\end{array}
\right], \; \mathbf{x}_3 = \dfrac{1}{2b}
\left[
\begin{array}{c}
[b(n-3) - (n-4)]\mathds{1}_2 \\
\hline
-(b-1) \mathds{1}_{n-3}
\end{array}
\right]\mbox{ and }\\
  x &=&\dfrac{1}{2b}\left[ -n(b-1)^2 + (3b^2-10b+6) \right].
\end{eqnarray*}
\end{small}
Let $Y = D(T_n^{(b)})X = (Y_{ij})$ be a $(b+1)$-dimensional block matrix, where $Y_{ij}$ are block matrices of 
conformal order, given by
\begin{small}
\begin{equation}\label{eqn:Y}
Y_{ij}=
\begin{cases}
D_1X_1 +(b-1)D_2X_2 + \mathbf{d}_3\mathbf{x}_3^t, & \text{if} \ i = j,\ i=1,2,...,b,  \\
D_1X_2+D_2X_1 +(b-2)D_2X_2+\mathbf{d}_3\mathbf{x}_3^t, &  \text{if} \ i \neq j,\  i,j = 1,2,...,b, \\
\mathbf{d}_3^tX_1+(b-1)\mathbf{d}_3^tX_2, & \text{if} \ i=b+1,\ j=1,2,...,b,\\
D_1\mathbf{x}_3+(b-1)D_2\mathbf{x}_3+\mathbf{d}_3x, & \text{if} \ j=b+1,\ i=1,2,...,b,\\
b \mathbf{d}_3^t\mathbf{x}_3, & \text{if} \ i=j=b+1,
\end{cases}
\end{equation}
\end{small}
$D_1, D_2$ and $\mathbf{d}_3$ are defined as in Eqn~(\ref{eqn:R_Tn_b}).

Since the steps are similar for Case $1 (n = 3)$ and Case $2 (n \geq 4)$, we only discuss the latter case. For simplicity, we compute the above five cases of block matrices $Y_{ij}$ in five different steps.

\noindent {\textbf{Step 1 : }}$\underline{Y_{ij}, \mbox{ for } i = j,\ i=1,2,...,b}$
\begin{fleqn}
\begin{flalign*}
 \mbox{Note that}  D_1X_1  &= \dfrac{1}{2b(n-6)}  \left[
                                    \begin{array}{c|c}
                                       (n-3)(n-4-2b)J_2 + 2b(n-6)I_2     & (b-1)J_{2 \times (n-3)}\\
                                       \hline
                                       -2b(n-6)J_{(n-3) \times 2} & 2b(n-6)I_{n-3}
                                    \end{array}
                                 \right],\\
  D_2X_2  &= \dfrac{1}{2b(n-6)} \left[
                                   \begin{array}{c|c}
                                        -(n-4)(n-7)J_2                & (n-7)J_{2 \times (n-3)}\\
                                        \hline
                                       -2(n-4)(n-6)J_{(n-3) \times 2} & 2(n-6)J_{n-3}
                                    \end{array}
                                 \right] \mbox{ and } \\
 \mathbf{d}_3\mathbf{x}_3^t  &= \dfrac{1}{2b} \left[
                                                  \begin{array}{c|c}
                                                     [b(n-3)-(n-4)]J_2 & -(b-1)J_{2 \times (n-3)}\\
                                                     \hline
                                                     2[b(n-3)-(n-4)]J_{(n-3) \times 2} & -2(b-1)J_{n-3}
                                                  \end{array}
                                                \right].
\end{flalign*}
\end{fleqn}
Thus, for $ i=1,2,...,b$, and $i = j$,  we have
$Y_{ij} = D_1X_1+(b-1)D_2X_2+\mathbf{d}_3\mathbf{x}_3^t=I_{n-1}.$

\noindent {\textbf{Step 2 : }} $\underline{Y_{ij},  \mbox{ for } i \neq j,\  i,j = 1,2,...,b}$
\begin{fleqn}
\begin{flalign*}
 \mbox{Now, } D_1X_2 &= \dfrac{1}{2b(n-6)} \left[
\begin{array}{c|c}
(n-4)J_2 & -J_{2 \times (n-3)}\\
\hline
\textbf{0}_{(n-3) \times 2} & \textbf{0}_{n-3}
\end{array}
\right] \mbox{ and }  \\
D_2X_1 &= \dfrac{1}{2b(n-6)} \left[
\begin{array}{c|c}
[-2b(n-5)-(n-4)(n-7)]J_2 & (n+b-7)J_{2 \times (n-3)}\\
\hline
-2(n-6)(n-4+b)J_{(n-3) \times 2} & 2(n-6)J_{n-3}
\end{array}
\right].
\end{flalign*}
\end{fleqn}
Further, for $ i,j=1,2,...,b$, and $i\neq j$,  substituting $D_2X_2$ and $\mathbf{d}_3\mathbf{x}_3^t$ from Step~1, we get
$$Y_{i,j}=D_1X_2+D_2X_1 +(b-2)D_2X_2+\mathbf{d}_3\mathbf{x}_3^t = \textbf{0}_{n-1}.$$

\noindent {\textbf{Step 3 : }} $\underline{Y_{ij},  \mbox{ for }  i=b+1,\  j = 1,2,...,b}$\\

\noindent Note that $\mathbf{d}_3^tX_1 = \dfrac{(b-1)}{b(n-6)} \left[
\begin{array}{c|c}
-(n-4)\mathds{1}_2^t & \mathds{1}_{n-3}^t\\
\end{array}
\right]$, $\mathbf{d}_3^tX_2 = \dfrac{1}{b(n-6)} \left[
\begin{array}{c|c}
(n-4)\mathds{1}_2^t & -\mathds{1}_{n-3}^t\\
\end{array}
\right]$, and hence for $i=b+1,\  j = 1,2,...,b$, we have $Y_{ij}=\mathbf{d}_3^tX_1+(b-1)\mathbf{d}_3^tX_2 = \textbf{0}_{1 \times (n-1)}.$

\noindent {\textbf{Step 4 : }} $\underline{Y_{ij},  \mbox{ for }  j=b+1,\  i = 1,2,...,b}$\\

\noindent Since $D_1\mathbf{x}_3 = \dfrac{1}{2b} \left[
\begin{array}{c}
\mathds{1}_2\\
\hline
2b\mathds{1}_{n-3}
\end{array}
\right]$,  $\mathbf{d}_3x = \dfrac{1}{2b}\left[ -n(b-1)^2 + (3b^2-10b+6) \right]\left[
\begin{array}{c}
\mathds{1}_2 \\
\hline
2\mathds{1}_{n-3}
\end{array}
\right]$,\\
and $D_2\mathbf{x}_3 = \dfrac{1}{2b} \left[
\begin{array}{c}
[b(n-3)-(n-7)]\mathds{1}_2\\
\hline
[2b(n-3)-2(n-6)]\mathds{1}_{n-3}
\end{array}
\right]$, we have for $j=b+1,\  i = 1,2,...,b,$ 

$$Y_{ij}=D_1\mathbf{x}_3+(b-1)D_2\mathbf{x}_3+\mathbf{d}_3x = \textbf{0}_{(n-1) \times 1}.$$

\noindent {\textbf{Step 5: }} $\underline{Y_{ij},  \mbox{ for } i= j=b+1}$\\

\noindent Since,  $ \mathbf{d}_3^t \mathbf{x}_3 = \dfrac{1}{2b} \left[
\begin{array}{c|c}
\mathds{1}_2^t & 2\mathds{1}_{n-3}^t\\
\end{array}
\right]
\left[
\begin{array}{c}
[b(n-3) - (n-4)]\mathds{1}_2 \\
\hline
-(b-1) \mathds{1}_{n-3}
\end{array}
\right]=\dfrac{1}{b}$ , so
$Y_{ij}=b\mathbf{d}_3^t\mathbf{x}_3 = 1,  \mbox{ for }  i= j=b+1.$

From the above fives steps and Eqn~(\ref{eqn:Y}), we get $Y=I$, the identity matrix, and hence the desired result follows.
\end{proof}


In the next section, we study eigenvalues of some principal submatrices of the matrix $\mathcal{R}(T_n^{(b)})$ defined in Eqn~(\ref{eqn:R_Tn_b}).

\section{Some properties of the matrix $\mathcal{R}(T_n^{(b)})$}\label{sec:Eigen}

In this section, we will study spectral properties of matrices related to $\mathcal{R}(T_n^{(b)})$. We compute the eigenvalues of certain principal submatrices of $\mathcal{R}(T_n^{(b)})$. To the best of our knowledge, this has not been attempted so far, although our aim is not to compare the spectral properties of $\mathcal{R}(T_n^{(b)})$ through this method. For simplicity, we will write $\mathcal{R}$ instead of $\mathcal{R}(T_n^{(b)})$. Before proceeding further we first fix a few notions which will appear subsequently. Recall that, the vertex set of $T_n^{(b)}$ is given by
\begin{flalign}\label{eqn:vertex}
 V(T_n^{(b)})=   & \{(k-1)(n-1)+i : k=1,2,\ldots, b  \mbox{ and } i=1,2,\ldots, n-1 \} \cup \{ b(n-1)+1\}.&&
\end{flalign}

Let $c$ denote the central cut vertex; that is, $c=b(n-1)+1$. 
Note that, the vertex set $V(T_n^{(b)})$ can be partitioned  into  $B$, $N$ and $\{c\}$,
where the set $B$ is the set of base vertices of each block
$T_n^{(b)}$  and $N$ is the set of non-base vertices of each block   $T_n^{(b)}$ excluding the central 
cut vertex.  For $n \ge 3$ and $b\ge 2$,
\begin{eqnarray*}
B&= & \{(k-1)(n-1)+i : k=1,2,\ldots, b  \mbox{ and } i=1,2 \},  \nonumber \\
N           & = &   \{(k-1)(n-1)+i : k=1,2,\ldots, b  \mbox{ and } i=3,4,\ldots, n-1 \}.
\end{eqnarray*}

\begin{ex}
From Figure~\ref{fig:T-n-b}, it can be seen \textup{(i)} For $T_3^{(4)}$, $B = \{1,2,3,4,5,6,7,8\}$ and $N$ is the empty set. \textup{(ii)} For $T_5^{(2)}$, $B =  \{1,2,5,6\}$ and $N = \{3,4,7,8\}$. \textup{(iii)} For $T_6^{(3)}$, $B = \{1,2,6,7,11,12\}$ and $N = \{3,4,5,8,9,10,13,14,15\}$.
\end{ex}

In what follows, we find the eigenvalues of the matrices $\mathcal{R}[B]$, $\mathcal{R}[N]$ and 
$\mathcal{R}[N\cup \{c\}]$, respectively,  where $\mathcal{R}[B]$, $\mathcal{R}[N]$ and 
$\mathcal{R}[N\cup \{c\}]$ are the principal submatrices of $\mathcal{R}$ corresponding the vertices coming from $B$, $N$ and $N\cup \{c\}$, respectively. We begin by computing the eigenvalues of $\mathcal{R}[B]$.

\begin{theorem}\label{Thm:S}
Let $n\ge 3$ and  $\mathcal{R}[B]$ be the principal submatrix of $\mathcal{R}$ corresponding to the 
base vertices $B$. The eigenvalues of $\mathcal{R}[B]$ are given by
\begin{small}
\[
\begin{array}{|c|c|}
\hline
\text{Eigenvalues} & \text{Multiplicity} \\
\hline
-(n-2)(n-6)b & 1 \\
\hline
(n-2)(n-6)b & b \\
\hline
(n-2)(n-4)b &(b-1)  \\
\hline
\end{array}
\]
\end{small}
\end{theorem}

\begin{proof}
Let $\mathcal{B} = \mathcal{R}[B] -xI = (s_{ij})$.  Let us label the block-rows and block-columns of the matrix $\mathcal{R}[B]$ by, $R_1,\cdots, R_b$ and $C_1,\cdots, C_b$, respectively. The following algorithm helps in determining the characteristics polynomial of $\mathcal{B}$.

\begin{itemizeReduced}
\item[1.] First do $R_1 \to \sum_{i=1}^b R_i$, followed by $C_j \to C_j - C_1$ for $j = 2, \cdots, b$.
\item[2.] For each of the block-rows, add the first to the second; for each of the block-columns, add the second 
one to the first one.
\end{itemizeReduced}

After applying the above steps to $\mathcal{B}$, we get a matrix $\widetilde{\mathcal{B}}$ in the 
following block form:
\begin{small}
\[
\widetilde{\mathcal{B}} =
\left[
\begin{array}{c|c|c|c}
\mathcal{B}_1 & \mathbf{0} &\cdots &\mathbf{0}  \\
\hline
* & \mathcal{B}_2 &\cdots &\mathbf{0}  \\
\hline
\vdots & \vdots & \ddots & \vdots  \\
\hline
* & * &\cdots &\mathcal{B}_2 \\
\end{array}
\right], \ \mbox{ where}
\]
\end{small}

\begin{small}
\begin{align*}
 \quad \mathcal{B}_1 = & \left[
\begin{array}{c|c}
-x-(n-2)(n-6)b & 0 \\
\midrule
* & -x+(n-2)(n-6)b \\
\end{array}
\right] \mbox{ and }\\
 \mathcal{B}_2 = & \left[
\begin{array}{c|c}
-x+(n-2)(n-4)b & 0\\
\midrule
* & -x+(n-2)(n-6)b\\
\end{array}
\right].&&
\end{align*}
\end{small}

Since both $\mathcal{B}_1  \ \& \ \mathcal{B}_2$ are lower triangular matrices, so is the matrix 
$\widetilde{\mathcal{B}}$. Therefore, the characteristic polynomial of $\mathcal{R}[B]$ equals 
$\det \mathcal{B} = \det \widetilde{\mathcal{B}}$ and is given by
$$\chi_{\mathcal{R}[B]}(x)=  (x+(n-2)(n-6)b)(x-(n-2)(n-6)b)^b(x-(n-2)(n-4)b)^{b-1}.$$ The desired result 
follows from this.
\end{proof}

We conclude by computing the eigenvalues of the principal submatrix $\mathcal{R}[N]$.

\begin{theorem}\label{thm:ev-N-n-ge-4}
Let $n \ge 4$ and $\mathcal{R}[N]$ be the principal submatrix of $\mathcal{R}$ corresponding the vertex 
set $N$. Then, the eigenvalues of $\mathcal{R}[N]$ are:
\begin{small}
\[
\begin{array}{|c|c|}
\hline
\text{Eigenvalues} & \text{Multiplicity} \\
\hline
-(n-4)(n-6)b & 1 \\
\hline
(2n-9)b & (b-1) \\
\hline
(n-6)b &(n-4)b  \\
\hline
\end{array}
\]
\end{small}
\end{theorem}

\begin{proof}
For $n \ge 4$,  $\mathcal{R}[N]$ is an $m \times m $ matrix, where $m= b(n-3)$.  
Let $\widetilde{\mathcal{N}} = \mathcal{R}[N] -xI = (t_{ij})$ and consider the permutation matrix $P_{\sigma}$ 
corresponding to the permutation $$\sigma= \prod_{k=1}^{b} ((k-1)(n-3)+1,k(n-3)).$$ Let us label the block-rows 
and block-columns of the matrix $\mathcal{R}[N]$ by $R_1,\cdots, R_b$ and $C_1,\cdots, C_b$, respectively.  We provide the following steps to get a suitable form of the matrix to calculate the eigenvalues.

\begin{itemizeReduced}
\item[1.] First do $R_1 \to \sum_{i=1}^b R_i$, followed by $C_j \to C_j - C_1$ for $j = 2, \cdots b$.
\item[2.] If $n>4$ subtract the last column in each of block-columns from the remaining columns.
\item[3.] For each block-rows add all the rows to the last row.
\item[4.] Rearrange the rows and columns by pre-multiplying and post-multiplying by the permutation matrix $P_{\sigma}$.
\end{itemizeReduced}

Proceeding with the above steps, the resulting matrix $\widetilde{\mathcal{N}}$ has the following block matrix form
\begin{small}
 \[
\widetilde{\mathcal{N}} = \left[
\begin{array}{c|c|c|c}
\widetilde{\mathcal{N}}_1 & \mathbf{0}                &\cdots &\mathbf{0} \\
\hline
*                         & \widetilde{\mathcal{N}}_2 &\cdots &\mathbf{0}\\
\hline
\vdots                    & \vdots                    & \ddots & \vdots \\
\hline
*                         & *                         &\cdots &\widetilde{\mathcal{N}}_2\\
\end{array}
\right], \ \mbox{where}
\] 
\end{small}

\noindent $\underline{\textbf{Case 1:}} \, $($n = 4$) $\widetilde{\mathcal{N}}_1= -x \mbox{ and }\widetilde{\mathcal{N}}_2= -(x+b).$

\noindent $\underline{\textbf{Case 2:}} \, $($n \ge 5$)

\begin{align*}
\widetilde{\mathcal{N}}_1 &= \left[
\begin{array}{c|c}
-x-(n-4)(n-6)b  & \mathbf{0} \\
\hline
*               & (-x+(n-6)b)I_{n-4}\\
\end{array}
\right] \mbox{ and } \\
\widetilde{\mathcal{N}}_2 &=\left[
\begin{array}{c|c}
-x+(2n-9)b & \mathbf{0} \\
\hline
 *         &(-x+(n-6)b)I_{n-4}  \\
\end{array}
\right].
\end{align*}

It is now easy to compute the characteristic polynomial (and hence its roots) of $\mathcal{R}[N]$. 
\end{proof}

\begin{theorem}\label{thm:ev-N-c-n-ge-4}
Let $n \ge 4$ and $\mathcal{R}[N\cup \{c\}]$ be the principal submatrix of $\mathcal{R}$ corresponding the 
vertex set $N\cup \{c\}$. The eigenvalues of $\mathcal{R}[N \cup \{c\}]$ are $(2n-9)b$ with multiplicity $(b-1), 
\ (n-6)b$ with multiplicity $(n-4)b$ and the remaining two eigenvalues are the solutions to the quadratic 
polynomial $p(x)$, where
\begin{equation*}
p(x)=x^2 + (n-6)\left[(n-5)(b-1)^2+(n-4)b\right]x +
b(n-6)^2\left[(n-4)(n-5)(b-1)^2 - (n-3)b^2\right].
\end{equation*}
\end{theorem}
\begin{proof}
Note that, $\mathcal{R}[N\cup \{c\}]$ is an $m \times m$ matrix, where $m=b(n-3)+1$.  Let $\widetilde{\mathcal{N}} = \mathcal{R}[N \cup \{c\}] -xI$.  Applying  the algorithm in Theorem~\ref{thm:ev-N-n-ge-4} and  the permutation 
\begin{equation*}
\sigma=\begin{cases} 
(2,b+1), & \text{if } n=4,\\
(2,b(n-3)+1)\prod_{k=1}^{b} ((k-1)(n-3)+1,k(n-3)), & \text{if } n\geq 5,
\end{cases}
\end{equation*}
we get the desired result.
\end{proof}

%


\begin{small}

\end{small}

\end{document}